\documentclass[english]{amsart}

\usepackage{amsmath}
\usepackage{amssymb}
\usepackage{amsthm}
\usepackage{amscd}
\usepackage{babel}
\usepackage[latin1]{inputenc}
\usepackage{graphicx}

\renewcommand{\subjclassname}{AMS \textup{2000} Mathematics Subject Classification:\ }

\newtheorem{teor}{Theorem}
\newtheorem{cor}{Corollary}
\newtheorem{prop}{Proposition}
\newtheorem{lem}{Lemma}
\newtheorem{defi}{Definition}
\newtheorem{con}{Conjecture}

\author{Jos\'{e} Mar\'{i}a Grau}
\address{Departamento de Matemáticas, Universidad de Oviedo\\ Avda. Calvo Sotelo, s/n, 33007 Oviedo, Spain}
\email{grau@uniovi.es}
\title{On $k$-Lehmer numbers}

\author{Antonio M. Oller-Marc\'{e}n}
\address{Centro Universitario de la Defensa\\ Ctra. Huesca s/n, 50090 Zaragoza, Spain} \email{oller@unizar.es}
\begin{document}
\maketitle
\begin{abstract}
Lehmer's totient problem consists of determining the set of positive integers $n$ such that $\varphi(n)|n-1$ where $\varphi$ is Euler's totient function. In this paper we introduce the concept of $k$-Lehmer number. A $k$-Lehmer number is a composite number such that $\varphi(n)|(n-1)^k$. The relation between $k$-Lehmer numbers and Carmichael numbers leads to a new characterization of Carmichael numbers and to some conjectures related to the distribution of Carmichael numbers which are also $k$-Lehmer numbers.
\end{abstract}
\subjclassname{11A25,11B99}

\section{Introduction}
Lehmer's totient problem asks about the existence of a composite number such that $\varphi(n)|(n-1)$, where $\varphi$ is Euler's totient function. Some authors denote these numbers by \emph{Lehmer numbers}.
In 1932, Lehmer (see \cite{LEH}) showed that every Lemher numbers $n$ must be odd and square-free, and that the number of distinct prime factors of $n$, $d(n)$, must satisfy $d(n)>6$. This bound was subsequently extended to $d(n)>10$. The current best result, due to Cohen and Hagis (see \cite{COH}), is that $n$ must have at least 14 prime factors and the biggest lower bound obtained for such numbers is $10^{30}$ (see \cite{PIN}). It is known that there are no Lehmer numbers in certain sets, such as the Fibonacci sequence (see \cite{LUC1}), the sequence of repunits in base $g$ for any $g\in [2, 1000]$ (see \cite{CIL}) or the Cullen numbers (see \cite{GRA}). In fact, no Lemher numbers are known up to date. For further results on this topic we refer the reader to \cite{BAN2,BAN1,LUC2,POM1,POM2}.

A \emph{Carmichael number} is a composite positive integer $n$ satisfying the congruence $b^{n-1}\equiv 1$ (mod $n$) for every integer $b$ relatively prime to $n$. Korselt (see \cite{KOR}) was the first to observe the basic properties of Carmichael numbers, the most important being the following characterization:

\begin{prop}[Korselt, 1899]
A composite number $n$ is a Carmichael number if and only if n is square-free, and for each prime $p$ dividing $n$, $p-1$ divides $n-1$.
\end{prop}

Nevertheless, Korselt did not find any example and it was Robert Carmichael in 1910 (see \cite{CAR1}) who found the first and smallest of such numbers (561) and hence the name ``Carmichael number'' (which was introduced by Beeger in \cite{BEE}). In the same paper Carmichael presents a function $\lambda$ defined in the following way:
\begin{itemize}
\item $\lambda(2)=1$, $\lambda(4)=2$.
\item $\lambda(2^k)=2^{k-2}$ for every $k\geq 3$.
\item $\lambda(p^k)=\varphi(p^k)$ for every odd prime $p$.
\item $\lambda(p_1^{k_1}\cdots p_m^{k_m})=\textrm{lcm}\left(\lambda(p_1^{k_1}),\dots,\lambda(p_m^{k_m})\right)$.
\end{itemize}
With this function he gave the following characterization:

\begin{prop}[Carmichael, 1910]
A composite number $n$ is a Carmichael number if and only if $\lambda(n)$ divides $(n-1)$.
\end{prop}

In 1994 Alford, Granville and Pomerance (see \cite{ALF}) answered in the affirmative the longstanding question whether there were infinitely many Carmichael numbers. From a more computational viewpoint, the paper \cite{GUN} gives an algorithm to construct large Carmichael numbers. In \cite{BAL} the distribution of certain types of Carmichael numbers is studied.

In this work we introduce the condition  $\varphi(n)| (n-1)^k$ (that we shall call \emph{$k$-Lehmer property} and the associated concept of $k$-Lehmer numbers. In the first section we give some properties of the sets $L_k$ (the set of numbers satisfying the $k$-Lehmer property) and $\displaystyle{L_{\infty}:=\bigcup_{k\geq 1} L_k}$, characterizing this latter set. In the second section we show that every Carmichael number is also a $k$-Lehmer number for some $k$. Finally, in the third section we use Chernick's formula to construct Camichael numbers in $L_k\setminus L_{k-1}$ and we give some related conjectures.

\section{A generalization of Lehmer's totient property}

Recall that a \emph{Lehmer number} is a composite integer $n$ such that $\varphi(n)|n-1$. Following this idea we present the definition below.

\begin{defi}
Given $k\in\mathbb{N}$, a $k$-Lehmer number is a composite integer $n$ such that $\varphi(n)|(n-1)^k$. If we denote by $L_k$ the set:
$$L_{k} := \{n \in \mathbb{N}\ |\  \varphi(n)|(n-1)^{k}\},$$
it is clear that $k$-Lehmer numbers are the composite elements of $L_k$.
\end{defi}

Once we have defined the family of sets $\{L_k\}_{k\geq 1}$ and since $L_k\subseteq L_{k+1}$ for every $k$, it makes sense to define a set $L_{\infty}$ in the following way:

$$L_{\infty} := \bigcup^{\infty}_{k=1}  L_{k}.$$
The set $L_{\infty}$ is easily characterized in the following proposition.

\begin{prop}\label{car}
$$L_{\infty}=\{n \in \mathbb{N}\ |\ \rm{rad}(\varphi(n))|n-1\}.$$
\end{prop}
\begin{proof}
Let $n\in L_{\infty}$. Then $n\in L_k$ for some $k\in\mathbb{N}$. Now, if $p$ is a prime dividing $\varphi(n)$, it follows that $p$ divides $(n-1)^k$ and, being prime, it also divides $n-1$. This proves that $\textrm{rad}(\varphi(n))|n-1$.

On the other hand, if $\textrm{rad}(\varphi(n)) | n-1$ it is clear that $\varphi(n) | (n-1)^k$ for some $k\in \mathbb{N}$. Thus $n\in L_k\subseteq L_{\infty}$ and the proof is complete. 
\end{proof}

Obviously, the composite elements of $L_{1}$ are precisely the Lehmer numbers and the Lehmer property asks whether $L_1$ contains composite numbers or not. Nevertheless, for all $k>1$,  $L_{k}$ always contains composite elements. For instance, the first few composite elements of $L_2$ are (sequence A173703 in OEIS):
$$\{561, 1105, 1729, 2465, 6601, 8481, 12801, 15841, 16705, 19345, 22321, 30889, 41041,\dots\}.$$

Observe that in the previous list of elements of $L_2$ there are no products of two distinct primes. We will now prove this fact, which is also true for Carmichael numbers. Observe that this property is no longer true for $L_3$ since, for instance, $15\in L_3$ and also the product of two Fermat primes lies in $L_{\infty}$.

In order to show that no product of two distinct odd primes lies in $L_2$ we will give a stronger result which determines when an integer of the form $n=pq$ (with $p\neq q$ odd primes) lies in a given $L_k$.

\begin{prop}
Let $p$ and $q$ be distinct odd primes and let $k\geq 2$. Put $p=2^ad\alpha+1$ and $q=2^bd\beta+1$ with $d$, $\alpha$, $\beta$ odd and $\gcd(\alpha,\beta)=1$. We can assume without loss of generality that $a\leq b$. Then $n=pq\in L_k$ if and only if $a+b\leq ka$ and $\alpha\beta | d^{k-2}$.
\end{prop}
\begin{proof}
By definition $pq\in L_k$ if and only if $\varphi(pq)=(p-1)(q-1)=2^{a+b}d^2\alpha\beta$ divides $(pq-1)^k=\left(2^{a+b}d^2\alpha\beta+2^ad\alpha+2^bd\beta\right)^k$. If we expand the latter using the multinomial theorem it easily follows that $pq\in L_k$ if and only if $2^{a+b}d^2\alpha\beta$ divides $2^{ka}d^k\alpha^k+2^{kb}d^k\beta^k=2^{ka}d^k\left(\alpha^k+2^{k(b-a)}\beta^k\right)$.

Now, if $a\neq b$ observe that $\left(\alpha^k+2^{k(b-a)}\beta^k\right)$ is odd and, since $\gcd(\alpha,\beta)=1$, it follows that $\gcd(\alpha,\alpha^k+2^{k(b-a)}\beta^k)=\gcd(\beta,\alpha^k+2^{k(b-a)}\beta^k)=1$. This implies that $pq\in L_k$ if and only if $a+b\leq ka$ and $\alpha\beta$ divides $d^{k-2}$ as claimed.

If $a=b$ then $pq\in L_k$ if and only if $\alpha\beta$ divides $d^{k-2}\left(\alpha^k+\beta^k\right)$ and the result follows like in the previous case. Observe that in this case the condition $a+b\leq ka$ is vacuous since $k\geq 2$.
\end{proof}

\begin{cor}
If $p$ and $q$ are distinct odd primes, then $pq\not\in L_2$.
\end{cor}
\begin{proof}
By the previous proposition and using the same notation, $pq\in L_2$ if and only if $a+b\leq 2a$ and $\alpha\beta$ divides 1. Since $a\leq b$ the first condition implies that $a=b$ and the second condition implies that $\alpha=\beta=1$. Consequently $p=q$, a contradiction.
\end{proof}

It would be interesting to find an algorithm to construct elements in a given $L_k$. The easiest step in this direction, using similar ideas to those in Proposition 6, is given in the following result.

\begin{prop}
Let $p_r=2^r\cdot 3+1$. If $p_N$ and $p_M$ are primes and $M-N$ is odd, then $n=p_Np_M\in L_{K}$ for $K=\min\{k\ |\ kN\geq M+N\}$ and $n\not\in L_{K-1}$.
\end{prop}

We will end this section with a table showing some values of the counting function for some $L_k$. If
$$C_k(X):=\sharp \{n \in L_k: x\leq X\},$$
we have the following data:
\begin{center}
\begin{tabular}{|c|c|c|c|c|c|c|c|c|}
  \hline
  n & 1 & 2 & 3 &4& 5 & 6 &7 &8 \\ \hline
  $C_2(10^n)$ & 5& 26 & 170 & 1236& 9613 & 78535 &664667 &5761621\\\hline
  $C_3(10^n)$ & 5& 29& 179& 1266& 9714 & 78841 & 665538 &5763967 \\\hline
  $C_4(10^n)$ & 5& 29& 182& 1281& 9784 & 79077 & 666390 & 5766571\\\hline
  $C_5(10^n)$ & 5& 30& 184& 1303& 9861& 79346 & 667282  &5769413\\\hline
$ C_\infty(10^n)$& 5& 30& 188& 1333& 10015& 80058 &670225 & 5780785 \\
  \hline
\end{tabular}
\end{center}

This table leads us to the following conjecture about the asymptotic behavior of $C_k(X)$.
\begin{con} 
For every $k>1$, the asymptotic behavior of $C_k$ does not depend on $k$ and, in particular:
$$C_k(x)\in \mathcal{O}(\frac{x}{\log\log x})$$
\end{con}

\section{Relation with Carmichael numbers}

This section will study the relation of $L_{\infty}$ with square-free integers and with Carmichael numbers. The characterization of $L_{\infty}$ given in Proposition \ref{car} allows us to present the following straightforward lemma which, in particular, implies that $L_\infty $ has zero asymptotic density (like the set of cyclic numbers, whose counting function is $\mathcal{O}(\frac{x}{\log\log\log x})$, see \cite{ERDOS}).

\begin{lem}\label{lem}
If $n\in L_{\infty}$, then $n$ is a cyclic number; i.e., $\gcd(n,\varphi(n))=1$ and consequently square-free.
\end{lem}

Recall that every Lehmer number (if any exists) must be a Carmichael number. The converse is clearly false but, nevertheless, we can see that every Carmichael number is a $k$-Lehmer number for some $k\in\mathbb{N}$.

\begin{prop}\label{infi}
If $n$ is a Carmichael number, then $n\in L_{\infty}$
\end{prop}
\begin{proof}
Let $n$ be a Carmichael number. By Korselt's criterion $n=p_1\cdots p_m$ and $p_i-1$ divides $n-1$ for every $i\in\{1,\dots,m\}$. We have that $\varphi(n)=(p_1-1)\cdots (p_m-1)$ and we can put $rad(\varphi(n))=q_1\cdots q_r$ with $q_j$ distinct primes. Now let $j\in\{1,\dots,r\}$, since $q_j$ divides $\varphi(n)$ it follows that $q_j$ divides $p_i-1$ for some $i\in\{1,\dots,m\}$ and also that $q_j$ divides $n-1$. This implies that $rad(\varphi(n))$ divides $n-1$ and the result follows.
\end{proof}

The two previous results lead to a characterization of Carmichael numbers  which slightly modifies Korselt's criterion. Namely, we have the following result.

\begin{teor}
A composite number $n$ is a Carmichael number if and only if  $rad(\varphi(n))$ divides $n-1$ and $p-1$ divides $n-1$ for every $p$ prime divisor of $n$.
\end{teor}
\begin{proof}
We have already seen in Proposition \ref{infi} that if $n$ is a Carmichael number, then $rad(\varphi(n))$ divides $n-1$ and, by Korselt's criterion $p-1$ divides $n-1$ for every $p$ prime divisor of $n$.

Conversely, if $rad(\varphi(n))$ divides $n-1$ then by Lemma \ref{lem} we have that $n$ is square-free so it is enough to apply Korselt's criterion again.
\end{proof}

The set $L_\infty$ not only contains every Carmichael numbers (which are absolute pseudoprimes) but all the elements of $L_{\infty}$ are Fermat pseudoprimes to some base $b$ with $1<b<n-1$. In fact, we have:

\begin{prop} 
Let $n\in L_\infty$ be a composite integer and let $b$ be an integer such that $b\equiv 2^{\frac{\varphi(n)}{rad(\varphi(n))}}$ (mod $n$). Then $n$ is a Fermat pseudoprime to base $b$.
\end{prop}
\begin{proof}
Since $n\in L_{\infty}$, it is odd and $\textrm{rad}(\varphi(n))$ divides $n-1$. Thus:
$$b^{n-1}\equiv 2^{\frac{\varphi(n)(n-1)}{\rm{rad}(\varphi(n))}}=2^{\varphi(n)\frac{n-1}{\rm{rad}(\varphi(n))}}\equiv 1\ \textrm{(mod $n$)}.$$
\end{proof}

\section{Carmichael numbers in $L_k \backslash L_{k-1} $. some conjectures.}

Recall the list of elements from $L_2$ given in the previous section:
$$L_{2}=\{\textbf{561}, \textbf{1105}, \textbf{1729}, \textbf{2465}, \textbf{6601}, 8481, 12801, \textbf{15841}, 16705, 19345, 22321, 30889, 41041\dots\}.$$
Here, numbers in boldface are Carmichael numbers. Observe that not every Carmichael number lies in $L_2$, the smallest absent one being 2821. Although 2821 doe not lie in $L_2$ in is easily seen that 2821 lies in $L_3$.

It would be interesting to study the way in that Carmichael numbers are distributed among the sets $L_k$. In this section we will present a first result in this direction together with some conjectures.

Recall Chernick's formula (see \cite{CHE}):
$$U_k(m)=(6m + 1)(12m + 1)\prod_{i=1}^{k-2}(9\cdot 2^im+1).$$ 
$U_k(m)$ is a Carmichael number provided all the factors are prime and $2^{k-4}$ divides $m$. Whether this formula produces an infinity quantity of Carmichael numbers is still not known, but we will see that it behaves quite nicely with respect to our sets $L_k$.

\begin{prop}\label{cher}
Let $k>2$. If $(6m + 1)$, $(12m + 1)$ and $(9\cdot 2^i m+1)$ for $i=1,\dots,k-2$ are primes and $m\equiv 0$ (mod $2^{k-4}$) is not a power of 2, then $U_k(m) \in  L_{k}\setminus L_{k-1}$.
\end{prop}
\begin{proof}
It can be easily seen by induction (we give no details) that $U_k(n)-1=2^23^2m\left(2^{k-3}+\sum_{i=1}^{k-1} a_im^i\right)$. On the other hand we have that $\varphi\left(U_k(m)\right)=2^{\frac{k^2-3k+8}{2}}3^{2k-2}m^k$.

Let us see that $U_k(m)\in L_k$. To do so we study two cases:
\begin{itemize}
\item Case 1: $3\leq k\leq 5$.

In this case $\frac{k^2-3k+8}{2}<2k$ and, consequently:
$$\varphi\left(U_k(m)\right)=2^{\frac{k^2-3k+8}{2}}3^{2k-2}m^k\ \big |\ (2^23^2m)^k\ \big |\ (U_k(m)-1)^k.$$
\item Case 2: $k\geq 6$.

Since $2^{k-4}$ divides $m$ we have that $2^{k-4}$ divides $2^{k-3}+\sum_{i=1}^{k-1} a_im^i$. Consequently, since $k(k-4)<\frac{k^2-3k+8}{2}$ in this case, we get that:
$$\varphi\left(U_k(m)\right)=2^{\frac{k^2-3k+8}{2}}3^{2k-2}m^k\ \big |\ 2^{k(k-4)}3^{2k-2}m^k\ \big |\ (U_k(m)-1)^k.$$
\end{itemize}

Now, we will see that $U_k(m)\not\in L_{k-1}$. Since $U_k(m)-1)^{k-1}=2^{2k-2}3^{2k-2}(2^{k-3}+\left(\sum_{i=1}^{k-1} a_im^i\right)^{k-1}$, it follows that $U_k(m)\in L_{k-1}$ if and only if $2^{\frac{(k-3)(k-4)}{2}}m$ divides $\left(\sum_{i=1}^{k-1} a_im^i\right)^{k-1}$. If we put $m=2^hm'$ with $m'$ odd this latter condition implies that $m'|2^{k-3}{k-1}$ which is clearly a contradiction because $m$ is not a power of 2. This ends the proof.
\end{proof}

This result motivates the following conjecture.

\begin{con}
For every $k\in \mathbb{N}$, $L_{k+1}\setminus L_k$ contains infinitely many Carmichael numbers.
\end{con}

Now, given $k\in\mathbb{N}$, let us denote by $\alpha(k)$ the smallest Carmichael number $n$ such that $n \not \in L_k$:
$$\alpha(k)=\min\{n\ |\ \textrm{$n$ is a Carmichael number,}\ n \not\in L_k\}.$$
The following table presents the first few elements of this sequence (A207080 in OEIS):
$$\begin{array}{ccc}
  k & \alpha(k)  &\textrm{Prime Factors}\\
  1 & 561 &3 \\
  2 & 2821 &3 \\
  3 & 838201 &4\\
  4 & 41471521&5  \\
  5 & 45496270561&6 \\
   6 & 776388344641&7 \\
      7 & 344361421401361 &8\\
     8 & 375097930710820681&9 \\
      9 & 330019822807208371201&10 \\
\end{array}$$

These observations motivate the following conjectures which close the paper:

\begin{con}
For every $k\in \mathbb{N}$, $\alpha(k)\in L_{k+1}$.
\end{con}

\begin{con}
For every $2<k\in \mathbb{N}$, $\alpha(k)$ has $k+1$ prime factors.
\end{con}

\bibliography{./refcarmi}

\begin{thebibliography}{10}

\bibitem{ALF}
W.~R. Alford, Andrew Granville, and Carl Pomerance.
\newblock There are infinitely many {C}armichael numbers.
\newblock {\em Ann. of Math. (2)}, 139(3):703--722, 1994.

\bibitem{BAL}
R.~Balasubramanian and S.~V. Nagaraj.
\newblock Density of {C}armichael numbers with three prime factors.
\newblock {\em Math. Comp.}, 66(220):1705--1708, 1997.

\bibitem{BAN2}
William~D. Banks, Ahmet~M. G{\"u}lo{\u{g}}lu, and C.~Wesley Nevans.
\newblock On the congruence {$N\equiv A\pmod{\phi(N)}$}.
\newblock {\em Integers}, 8:A59, 8, 2008.

\bibitem{BAN1}
William~D. Banks and Florian Luca.
\newblock Composite integers {$n$} for which {$\phi(n)\mid n-1$}.
\newblock {\em Acta Math. Sin. (Engl. Ser.)}, 23(10):1915--1918, 2007.

\bibitem{BEE}
N.~G. W.~H. Beeger.
\newblock On composite numbers {$n$} for which {$a\sp {n-1}\equiv ({\rm
  mod}\,n)$} for every {$a$} prime to {$n$}.
\newblock {\em Scripta Math.}, 16:133--135, 1950.

\bibitem{CAR1}
R.~D. Carmichael.
\newblock Note on a new number theory function.
\newblock {\em Bull. Amer. Math. Soc.}, 16(5):232--238, 1910.

\bibitem{CHE}
Jack Chernick.
\newblock On {F}ermat's simple theorem.
\newblock {\em Bull. Amer. Math. Soc.}, 45(4):269--274, 1939.

\bibitem{CIL}
Javier Cilleruelo and Florian Luca.
\newblock Repunit {L}ehmer numbers.
\newblock {\em Proc. Edinburgh Math. Soc.}, to appear.

\bibitem{COH}
G.~L. Cohen and P.~Hagis, Jr.
\newblock On the number of prime factors of {$n$} if {$\varphi (n)|(n-1)$}.
\newblock {\em Nieuw Arch. Wisk. (3)}, 28(2):177--185, 1980.

\bibitem{ERDOS}
P.~Erd{\"o}s.
\newblock Some asymptotic formulas in number theory.
\newblock {\em J. Indian Math. Soc. (N.S.)}, 12:75--78, 1948.

\bibitem{GRA}
José~María Grau and Florian Luca.
\newblock Cullen numbers with the {L}ehmer property.
\newblock {\em Proc. Amer. Math. Soc.}, 140(129--134), 2012.

\bibitem{KOR}
Korselt.
\newblock Problème chinois.
\newblock {\em L'intermédiaire des mathématiciens}, 6:142--143.

\bibitem{LEH}
D.~H. Lehmer.
\newblock On {E}uler's totient function.
\newblock {\em Bull. Amer. Math. Soc.}, 38(10):745--751, 1932.

\bibitem{GUN}
G{\"u}nter L{\"o}h and Wolfgang Niebuhr.
\newblock A new algorithm for constructing large {C}armichael numbers.
\newblock {\em Math. Comp.}, 65(214):823--836, 1996.

\bibitem{LUC1}
Florian Luca.
\newblock Fibonacci numbers with the {L}ehmer property.
\newblock {\em Bull. Pol. Acad. Sci. Math.}, 55(1):7--15, 2007.

\bibitem{LUC2}
Florian Luca and Carl Pomerance.
\newblock On composite integers $n$ for which $\varphi(n)|n-1$.
\newblock {\em Bol. Soc. Mat. Mexicana}, to appear.

\bibitem{PIN}
Richard~G.E. Pinch.
\newblock A note on {L}ehmer's totient problem.
\newblock {\em Poster presented in ANTS VII,
  http://www.math.tu-berlin.de/$\sim$kant/ants/Poster/Pinch\_Poster3.pdf}.

\bibitem{POM1}
Carl Pomerance.
\newblock On composite {$n$} for which {$\varphi (n)\mid n-1$}. {II}.
\newblock {\em Pacific J. Math.}, 69(1):177--186, 1977.

\bibitem{POM2}
Carl Pomerance.
\newblock On the distribution of amicable numbers.
\newblock {\em J. Reine Angew. Math.}, 293/294:217--222, 1977.

\end{thebibliography}
\bibliographystyle{plain}
\end{document}